\documentclass[11pt,a4paper]{amsart}

\title{Non-associative Frobenius algebras of type $E_7$}
\author{Jari Desmet}
\address{\parbox{\linewidth}{Ghent University \\ Department of Mathematics: Algebra and Geometry \\ Krijgslaan 281 -- S25 \\ 9000 Gent \\ Belgium}}
\email{\href{mailto:jari.desmet@ugent.be}{jari.desmet@ugent.be}}
\date{\today}
\makeatletter
\@namedef{subjclassname@2020}{%
  \textup{2020} Mathematics Subject Classification}
\makeatother

\usepackage[english]{babel}
\usepackage{amsmath,amssymb,amsthm}
\usepackage[utf8]{inputenc}
\usepackage{graphicx}
\usepackage{tikz}
\usetikzlibrary{calc}

\usepackage{arydshln}
\usepackage{tikzpagenodes}
\usepackage[shortlabels]{enumitem}
\setenumerate[0]{label={\rm (\roman*)}, leftmargin=*}
\usepackage{faktor}
\usepackage{mathtools}
\usepackage{breqn}
\usepackage{comment}

\usepackage[margin = 3cm]{geometry}
\usepackage[textsize=footnotesize,textwidth=15ex]{todonotes}

\usepackage{hyperref}
 
\usepackage[capitalise, noabbrev]{cleveref}

\theoremstyle{definition}
\newtheorem{definition}{Definition}[section]

\theoremstyle{plain}
\newtheorem*{theorem*}{Theorem}
\newtheorem{theorem}[definition]{Theorem}
\newtheorem{lemma}[definition]{Lemma}
\newtheorem{proposition}[definition]{Proposition}
\newtheorem{corollary}[definition]{Corollary}

\theoremstyle{remark}

\newtheorem{remark}[definition]{Remark}

\DeclareMathOperator{\spank}{span}
\DeclareMathOperator{\Lie}{Lie}
\DeclareMathOperator{\Der}{Der}

\DeclareMathOperator{\im}{Im}

\DeclareMathOperator{\gl}{\mathfrak{\gl}}
\DeclareMathOperator{\Oct}{\mathbb{O}}

\DeclareMathOperator{\Tr}{Tr}
\newcommand{\g}{\mathfrak{g}}

\DeclareMathOperator{\kar}{char}
\DeclareMathOperator{\ad}{ad}
\DeclareMathOperator{\Sq}{S}
\DeclareMathOperator{\End}{End}
\DeclareMathOperator{\Aut}{Aut}
\DeclareMathOperator{\GL}{GL}

\DeclareMathOperator{\SO}{SO}
\DeclareMathOperator{\Sp}{Sp}

\newcommand{\Trd}{\operatorname{Trd}_A}
\DeclareMathOperator{\Skew}{Skew}
\DeclareMathOperator{\Symm}{Sym}

\DeclareMathOperator{\Sand}{Sand}
\newcommand{\upi}[2]{#1^{(#2)}}
\newcommand{\B}[2]{\mathcal{B}(#1 , #2)}
\newcommand{\id}{\mathrm{id}}
\newcommand{\bgift}{(A,\sigma,\pi)}
\newcommand{\tbtmat}[4]{\begin{pmatrix}
	#1 & #2 \\
	#3 & #4
\end{pmatrix}}
\newcommand{\tbtsmallmat}[4]{\begin{psmallmatrix}
	#1 & #2 \\
	#3 & #4
\end{psmallmatrix}}

\let\phi\varphi

\numberwithin{equation}{section}

\begin{document}
\maketitle

\begin{abstract}
    Recently, Maurice Chayet and Skip Garibaldi introduced a class of commutative non-associative algebras.
    In previous work, we gave an explicit description of these algebras for groups of type $G_2,F_4$ and certain forms of $E_6$ in terms of octonion and  Albert algebras. In this paper, we extend this further by dealing with $E_7$ in terms of generalised Freudenthal triple systems.

\end{abstract}
\section{Introduction}
In \cite{chayet2020class}, Maurice Chayet and Skip Garibaldi define a construction with input any absolutely simple algebraic group $G$ over a field of large enough characteristic, and output a non-associative algebra $A(G)$ with a homomorphism $G\rightarrow \Aut(A(G))$. This construction applies to all absolutely simple algebraic groups, regardless of type and twisted form, and produces different algebras for different algebraic groups up to isogeny. Another construction for this algebra for the simply laced types was found independently by Tom De Medts and Michiel Van Couwenberghe in \cite{DMVC21}, where they prove complementary results about the structure of this algebra in the context of \emph{axial algebras}.

The algebra $A(G)$ is explicitly constructed from the symmetric square of the L\emph{ie algebra} $\Lie(G)$. However, it would be interesting to find constructions of these algebras that use other algebraic structures instead of the adjoint representation. This is hinted at in \cite[Proposition~10.5]{chayet2020class}, where it is shown $A(G)$ can be embedded in the endomorphism ring of the natural representation for types $A_2,G_2,F_4,E_6,E_7$, but other than type $A_2$ in \cite[Example~10.9]{chayet2020class}, Chayet and Garibaldi do not explore this further.

In \cite{desmetE6,jariprevious}, we investigated this embedding for the algebras of type $G_2,F_4$ and $E_6$, where the algebras were constructed using octonion algebras and Albert algebras. In this paper, we use \emph{Brown algebras} to construct $A(G)$ for certain forms $G$ of $E_7$ (specifically, those with trivial Tits algebras, i.e.\@ the groups of type $E_7$ we can associate a Brown algebra with). Afterwards we extend this to arbitrary forms of $E_7$, using lesser known objects called \emph{generalised Freudenthal triple systems} \cite{garibaldie7arbitrary}. Moreover, we provide some additional observations on the automorphism group of such algebras.

The paper follows a strategy similar to \cite{jariprevious}, but the concrete implementation becomes more convoluted because of the increased complexity of the Brown algebra versus the octonion algebra. There is also the additional issue of $E_7$ having possibly non-trivial Tits algebras.

\cref{sece7} is the largest section of the paper: we provide a model for the representation $V(\omega_6)$ using the Brown algebra, and describe all (commutative) $E_7$-equivariant multiplications on this model explicitly. Afterwards, in \cref{identificationag2}, we use these results to construct a model for $A(G)$.

In \cref{autgroup}, we observe that $A(G)$ has automorphism group $G$, when $G$ is an adjoint group of type $E_7$, and in \cref{sec:e7arbitrary}, we extend our description of $A(G)$ to arbitrary fields (under a mild assumption on the characteristic of the underlying field).

\subsection*{Acknowledgements}
	The author is supported by the FWO PhD mandate 1172422N. The author is grateful to his supervisor Tom De Medts, for guiding him through the process of writing this article.
\subsection*{Assumptions}
We use the same restrictions on the characteristic as in \cite{chayet2020class}, i.e.\@ $\kar k \geq h+2 = 20$ or $0$ with $h$ the Coxeter number of the Dynkin diagram $E_7$. In this case, the Weyl module of highest weight $\omega_6$ for $E_7$ is irreducible, as is the Weyl module of weight $2\tilde{\alpha}$, where $\tilde{\alpha}$ is the highest root, by \cite{lub01}. In particular, we can do character computations for these representations independently of the field $k$.

\section{Preliminaries}\label{preliminaries}
\subsection{Representation theory of absolutely simple groups and Lie algebras}
The irreducible representations of an absolutely simple split group over a field $k$ are classified by the dominant weights of the associated root system (\cite[Theorem~22.2]{Milne}). To any dominant weight $\lambda$ we can associate a so-called \emph{Weyl module} $V(\lambda)$. Under the characteristic assumption above, we know that the Weyl modules considered below for type $E_7$ are irreducible \cite{lub01}. We will often identify a representation by its associated dominant weight using labelling as in \cite[Plates I-IX, p.264-290]{Bourbaki46}. As mentioned in \cite[\S 7, p.10-11]{chayet2020class}, for non-split reductive groups there is a unique representation that becomes isomorphic to $V(\lambda)$ when base changing to $\bar{k}$ if $\lambda$ is contained in the root lattice and fixed by the Galois action as in \cite[Théorème 3.3]{Tits1971}. This representation will also be denoted by $V(\lambda)$. In the case of type $E_7$, this means we need the sum of the coefficients of $\omega_2,\omega_5$ and $\omega_7$ to be even. We will denote that representation by the same notation. For further discussion on irreducible representations of simple algebraic groups, see \cite{jantzenrepalggroups}.

\begin{remark}\label{reptheoryremark}
    Characters of representations of these Lie algebras (in particular the Weyl modules, over algebraically closed fields) can be computed using Sage \cite{sagemath}.
\end{remark}

\subsection{The Chayet-Garibaldi algebras}
Throughout this text, we will write $A(G)$ for the algebra constructed from a simple algebraic group $G$ in \cite{chayet2020class}.
For an introduction to these algebras, see \cite{chayet2020class}. As we will mostly be concerned with groups of type $E_7$, we give the necessary information for that case. We will denote the \emph{Killing form} on $\Lie(G)$ by $K$ and write $\ad X \in  \End(\Lie(G))$ for the map $Y\mapsto [X,Y]$ with $X\in \Lie(G)$, as is commonly done. 
\begin{theorem}[\cite{chayet2020class}]\label{modulestructure}
Let $G$ be an absolutely simple algebraic group of type $E_7$ over a field $k$ with $\kar k> 19$ or $0$, and let $\g = \Lie(G)$. Define
\begin{align*}
	S\colon \Sq^2\g &\to \End(\g) \\
	XY&\mapsto h^{\vee} \ad X \bullet \ad Y + \tfrac{1}{2}(X K(Y,\mathunderscore) + Y K(X, \mathunderscore) )\text{,}
\end{align*}
and
\begin{align*}
	\diamond \colon \im(S)\times \im(S) &\to \im(S) \\
	(S(AB), S(CD)) \mapsto \tfrac{h^{\vee}}{2} ( &
		S(A, (\ad C \bullet \ad D)B) + S((\ad C \bullet \ad D)A,B) \\
		&+ S(C, (\ad A \bullet \ad B)D) + S( (\ad A \bullet \ad B)C, D)\\
		&+ S([A,C][B,D])+S([A,D][B,C]) )\\
		+ \tfrac{1}{4}(& K(A,C)S(BD) + K(A,D)S(BC)\\
		&+ K(B,C)S(AD) + K(B,D)S(AC) )\text{.}
\end{align*}
Then $A(G) \coloneqq (\im(S),\diamond)$ is a well-defined simple unital commutative non-associative algebra with counit $\varepsilon(a) \coloneqq \tfrac{1}{\dim(G)}\Tr(a)$.
Moreover, $A(G)$ decomposes as a representation of $G$ into $A(G) \cong k\oplus V(\omega_6)$.	
\end{theorem} 

In the last section of \cite{chayet2020class}, the authors describe an embedding of these algebras into the endomorphism ring of a natural representation for the groups of type $A_2,G_2,F_4,E_6$ and $E_7$. It is this embedding that we will use to obtain new descriptions for the algebras in question. 
\begin{proposition}[{\cite[Proposition~10.5]{chayet2020class}}]\label{embedding}
	If $G$ has type $A_2,G_2,F_4,E_6$ or $E_7$ and $\pi\colon G \to \GL(W)$ is the natural irreducible representation of dimension $3,7,26,27$ or $56$ respectively, then the formula
	\begin{equation}\label{sigma}
		\sigma(S(XY)) =6h^{\vee} \pi(X)\bullet \pi(Y) - \tfrac{1}{2}K(X,Y)
	\end{equation}
	defines an injective $G$-equivariant linear map
	\[\sigma \colon A(G) \hookrightarrow \End(W)\text{.} \]
	Moreover, $\sigma$ maps the identity $\id_\g$ to the identity $\id_W$.	
\end{proposition}

The cases $A_2,G_2,F_4$ and $E_6$ have already been explored previously, see \cite{chayet2020class,desmetE6,jariprevious}. To deal with the case of $E_7$, we will model its $56$-dimensional irreducible representation as the \emph{Brown algebra}, an interesting family of algebras whose automorphism groups are simply connected of type $E_6$. The following subsection will briefly review Albert algebras and Brown algebras.

\subsection{The Albert algebras}\label{albertsubsection} We will only consider the Albert algebras over fields $k$ with $\kar k\neq 2,3$. We give a short overview in this subsection. 
	We regard the split Albert algebra as the hermitian matrices (with respect to the standard involution on $\Oct$) $(\mathcal{H}_3{(\Oct)},\bullet)$, where $\Oct$ denotes the split octonions over a field $k$. We write
\[ \mathbf{1}_1 \coloneqq \begin{bsmallmatrix}
	1&0&0\\
	0&0&0\\
	0&0&0
\end{bsmallmatrix},\mathbf{1}_2 \coloneqq \begin{bsmallmatrix}
	0&0&0\\
	0&1&0\\
	0&0&0
\end{bsmallmatrix},\mathbf{1}_3 \coloneqq \begin{bsmallmatrix}
	0&0&0\\
	0&0&0\\
	0&0&1
\end{bsmallmatrix}, \]
and for octonions $a,b,c\in \Oct$
\[ \upi{a}{1} \coloneqq \begin{bsmallmatrix}
	0&0&0\\
	0&0&a\\
	0&\bar{a}&0
\end{bsmallmatrix}, \upi{b}{2} \coloneqq \begin{bsmallmatrix}
	0&0&\bar{b}\\
	0&0&0\\
	b&0&0
\end{bsmallmatrix},\upi{c}{3} \coloneqq \begin{bsmallmatrix}
	0&c&0\\
	\bar{c}&0&0\\
	0&0&0
\end{bsmallmatrix}. \]	
	Any element of $\mathcal{H}_3{(\mathbb{O})}$ is then of the form
\[  
\begin{bsmallmatrix}
	\alpha_1 & c & \bar{b} \\
	\bar{c} & \alpha_2 & a \\
	b & \bar{a} & \alpha_3
\end{bsmallmatrix}
= \alpha_1\mathbf{1}_1+\alpha_2\mathbf{1}_2+\alpha_3\mathbf{1}_3+\upi{a}{1}+\upi{b}{2}+\upi{c}{3}\text{,}
\]

with $\alpha_1,\alpha_2,\alpha_3\in k$ and $a,b,c\in \mathbb{O}$. We will write $e = \mathbf{1}_1+\mathbf{1}_2+\mathbf{1}_3$ for the unit of an Albert algebra.
We then say a $k$-algebra $A$ is an Albert algebra if it becomes isomorphic to $(\mathcal{H}_3{(\Oct)},\bullet)$ over an algebraic closure of $k$. Alternatively, an Albert algebra is a $27$-dimensional central simple
exceptional Jordan algebra. We will denote the trace form by $\langle \cdot, \cdot \rangle$ and the linearisation of the determinant $\det$ by $\langle \cdot,\cdot,\cdot\rangle $. For an excellent introduction to Albert algebras, see \cite[Chapter 5]{SpringerVeldkamp}.

\begin{definition}\label{def:indicator}
Let $A$ be an Albert algebra over $k$.
\begin{enumerate}
	\item We write $L_x$ for the operator $L_x\colon J \to J \colon a \mapsto x\cdot a$.
	\item Let $k$ be algebraically closed and $i \in \{1,2,3\}$. We write $I_i$ for the projection operator (with respect to the trace form) onto the subspace $\{ \upi{a}{i}\mid a\in \Oct \}$.

\end{enumerate}
	\end{definition}
We will also need the so-called cross product, as it has a special relation with respect to groups of type $E_6$.
\begin{definition}[{\cite[Equation~5.16]{SpringerVeldkamp}}]
	For $x,y\in A$ elements of an Albert algebra, we define $x\times y\in A$ to be the unique element such that
	\[3\langle x,y,z\rangle  = \langle x\times y,z\rangle \]
	holds for all $z\in A$.
\end{definition}
\begin{remark}\label{rem:crossproductscale}
	Note that this definition differs from the one in \cite{garibaldie6e7} by a factor of $2$. This will become relevant in \cref{sec:brown}.
\end{remark}

\begin{lemma}\label{lem:crossproductproperties}\leavevmode Let $A$ be an Albert algebra over $k$.
	\begin{enumerate}
	
		\item The stabiliser $G$ of the determinant is a simply connected group of type $E_6$.
		\item Taking the transpose $^{\top}$ with respect to the bilinear form $\langle\cdot,\cdot\rangle$ and then inverting is an order $2$ outer automorphism of $G$.
		\item For $g\in G(k)$ and $x,y\in A$ we have $x^g\times y^g=(x\times y)^{g^{-\top}}$.
		\item We have $x\times y = xy - \tfrac{1}{2}\langle x,e \rangle y- \tfrac{1}{2}\langle y,e \rangle x- \tfrac{1}{2}\langle x,y \rangle e + \tfrac{1}{2}\langle x,e \rangle\langle y,e\rangle e$ for all $x,y\in A$. 
	\end{enumerate}
\end{lemma}
\begin{proof}
	See \cite[Section~7.3]{SpringerVeldkamp}.
\end{proof}
In fact, we have a converse of \cref{lem:crossproductproperties}(i) as well. For the definition of Tits algebras, see \cite[Section~27]{involutions}.
\begin{proposition}\label{prop:E6trivialtits}
	Any simply connected group of type $^1E_6$ with trivial Tits algebras occurs as the stabiliser of the determinant of an Albert algebra.
\end{proposition}
\begin{proof}
	See \cite[Theorem 1.4]{garibaldie6e7}.
\end{proof}
To compute the map $\sigma$ in \cref{embedding}, we will use some specific derivations of $E_6\subset E_7$. 
\begin{proposition}\label{prop:derivationsstructure}
	The space of derivations of the cubic form $\det$ on the Albert algebra $A$ is given by
	\[ \Der(\det) = \spank_k \{ L_a, [L_a,L_b] \mid a,b \in A \text{ such that } \langle a,e\rangle = \langle b,e\rangle = 0\}. \]
\end{proposition}
\begin{proof}
	This is \cite[Theorem~4.12]{nonassocalg}.
\end{proof}

\subsection{Brown algebras}\label{sec:brown}  Just like for the Albert algebras, we only consider Brown algebras over fields of characteristic different from $2$ or $3$. They occur naturally as structurable algebras of skew-dimension $1$, a concept that was first introduced and studied by Allison in \cite{defstructurablealg}, though Brown algebras themselves had been considered 15 years earlier in \cite{defbrown}. The relation of Brown algebras to so-called \emph{Freudenthal triple systems} has already proven to be a fruitful approach to understanding groups of type $E_7$ \cite{cayleydickson,garibaldie7arbitrary,garibaldie6e7, tomexceptionalq}.

\begin{definition}\label{def:struct algebra}
	Let $A$ be a $k$-algebra equipped with an involution $\sigma \colon x \mapsto \overline{x}$, i.e., a $k$-linear map satisfying $\overline{x.y}=\overline{y} .\overline{x}$ for all $x,y \in A$.
	Let
	\[ V_{x,y}(z)\coloneqq(x\overline y)z+(z\overline y)x-(z\overline x)y \eqqcolon \{x,y,z\} \]
	for all $x,y,z \in A$.
	If
		\[ [V_{x,y},V_{z,w}]=V_{V_{x,y}(z),w}-V_{z,V_{y,x}(w)} \]
	for all $x,y,z,w\in A$,
	then we call $A$ a \textit{structurable algebra}. We will call $A$ \emph{simple} if it has no nontrivial two-sided ideals, and central if its center is trivial (in the sense of \cite[p.\@ 134]{defstructurablealg}).
\end{definition}
 We construct the Brown algebra using the Albert algebra, as is commonly done.

\begin{definition}[Brown algebra]
	Let $J$ be an Albert algebra over $k$. We define the Brown algebra $\B{k}{J}$ to be the vector space $k\oplus J \oplus J\oplus k$ with multiplication\
	
		\[
\tbtmat{\alpha_1}{j_1}{j'_1}{\beta_1}
\tbtmat{\alpha_2}{j_2}{j'_2}{\beta_2} = 
\tbtmat{\alpha_1 \alpha_2 +  T(j_1, j'_2)}{\alpha_1 j_2 + \beta_2 j_1 +
 2j'_1 \times j'_2}{\alpha_2 j'_1 + \beta_1 j'_2 + 2j_1 \times
j_2}{\beta_1 \beta_2 +  T(j_2, j'_1)}.
\]
We endow $\B{k}{J}$ with the involution
$-$ given by
\[
\overline{\tbtmat{\alpha}{j}{j'}{\beta}} =
\tbtmat{\beta}{j}{j'}{\alpha}.
\]
We will call an algebra with involution $(A,-)$ over a field $k$ a Brown algebra if it becomes isomorphic to $\B{k}{J}$ for a certain Albert algebra $J$ after extending the base field.
\end{definition}
Note that the factor $2$ scaling the cross-product in the multiplication occurs because of \cref{rem:crossproductscale}. It turns out that one can associate a so-called Freudenthal triple system to any structurable algebra of skew-dimension $1$. For a definition, see \cite[Definition 3.1]{garibaldie7arbitrary}.
\begin{proposition}[{\cite[Proposition 2.18]{defstructurablealg}}]\label{prop:fts}
	Let $(A, -)$ be a structurable algebra of skew-dimension $1$ with unit e and let $s_0$ be a skew element (i.e.\@ $\overline{s_0} = -s_0$). Define $\omega\colon A\otimes A \to k$ and $t\colon A^{\otimes 3} \to A$ by the formulas
	\begin{equation}
		\omega(x,y)e \coloneqq (x\overline{y} - y\overline{x})s_0,
	\end{equation}
	\begin{equation}
		t(x,y,z) \coloneqq 2\{x,s_0y,z\} - \omega(y,z)x - \omega(y,x)z - \omega(x,z)y.
	\end{equation}
	Then $(A,\omega,t)$ has the structure of a Freudenthal triple system.
\end{proposition}
\begin{remark}\label{rem:sympformula}
	In the following, we will take $s_0 = \tbtsmallmat{1}{0}{0}{-1}$. In this case, the symplectic form $\omega$ takes the following form:
	\[\omega\left(\tbtsmallmat{\alpha_1}{j_1}{j^\prime_1}{\beta_1} ,\tbtsmallmat{\alpha_2}{j_2}{j^\prime_2}{\beta_2}\right) = \alpha_1\beta_2 - \alpha_2\beta_1 + \langle j_1,j^\prime_2\rangle -\langle j^\prime_1, j_2\rangle.  \]
\end{remark}
In \cite{garibaldie6e7}, Skip Garibaldi determines the automorphism groups of these structures.
\begin{proposition}[\cite{garibaldie6e7}]
	\begin{enumerate}
		\item The automorphism group of a Brown algebra is a simply connected group of type $E_6$ with trivial Tits algebras. All simply connected groups of type $E_6$ with trivial Tits algebras arise in this way.
		\item The automorphism group of the Freudenthal triple system associated to a Brown algebra as in \cref{prop:fts} is a simply connected group of type $E_7$ with trivial Tits algebras. All simply connected groups of type $E_7$ with trivial Tits algebras arise in this way.
	\end{enumerate}
\end{proposition}
\begin{remark}\label{rem:Browne6}
	We recall that the automorphisms of the Brown algebra $\B{k}{J}$ are of the form $\tbtsmallmat{\alpha}{j}{{j'}}{\beta}\mapsto \tbtsmallmat{\alpha}{j^g}{{j'}^{g^{-\top}}}{\beta}$, where $g$ is an invertible endomorphism of $J$ (as a vector space) stabilising the cubic norm of $J$, see \cite{garibaldie6e7}. 
\end{remark}

\section{The algebra of type $E_7$}\label{sece7}
Recall that
 the assumption on the characteristic in this section is $\kar k =0$ or $\kar k>19$. We will write $G$ for the group of type $E_7$ used to construct the algebra $A(G)$. We will assume $G$ to be simply connected, by passing to its universal cover (\cite[Remark 23.60]{Milne}) if necessary.
 
Throughout the rest of this section, we will assume $k$ is algebraically closed, with the exception of \cref{identificationag2}. The case of an arbitrary field will be discussed in \cref{sec:e7arbitrary}.

In this part, $W$ denotes the $56$-dimensional $G$-representation with highest weight $\omega_7$, and $V$ denotes the $1539$-dimensional representation with highest weight $\omega_6$. In this setting, the embedding from \cref{embedding} becomes, using \cite[Lemma~2.9]{chayet2020class},
\begin{equation}\label{embeddingg2}
	\sigma(S(XY)) = 108X\bullet Y - \tfrac{3}{2} \Tr(XY)\id_{W}\text{.} 
\end{equation}
In this formula, we do not write the embedding $\pi$ of the Lie algebra explicitly.

Our model for the algebra is constructed as a subrepresentation of the endomorphism ring of the Brown algebra, thus we will need to introduce notation for the operators we will use.
\begin{definition}\label{def:operators}
Let $x,y\in W$ and $i\in \{1,2,3\}$. 
	\begin{enumerate}	
		\item We write $x\wedge y$ for the operator $x\wedge y\colon W \to W \colon a \mapsto\omega (x,a) y - \omega (y, a ) x$,
		\item we write $L_{x,y}$ for the operator $L_{x,y}\colon W \to W \colon a \mapsto t(x,y,a)  $,
		\item we write $\mathbb{I}_i$ for the operator $\mathbb{I}_i \colon  W \to W \colon \tbtsmallmat{\alpha}{j}{j^\prime}{\beta} \mapsto \tbtsmallmat{0}{I_i(j)}{I_i(j^\prime)}{0}$, where $I_i$ is as in \cref{def:indicator}.
	\end{enumerate}
\end{definition}
\begin{lemma}\label{lem:trace}
	The trace of $x\wedge y$ is equal to $2\omega(x,y)$.
\end{lemma}
\begin{proof}
	If $x,y$ are linearly dependent, then $x\wedge y = 0$ and $\omega(x,y)=0$. If $x,y$ are linearly independent, we can extend $\{x,y\}$ to a basis, and explicitly compute the trace to be $2\omega(x,y)$, because the only contributions to the trace come from the image of $x$ and $y$.
\end{proof}

Note that we can use the natural action of $E_6$ (see \cref{rem:Browne6}) on the Brown algebra to find explicit derivations of the Brown algebra, and in particular of the symplectic form $\omega$ and triple product $t$.

\begin{definition}\label{def:derE7}
	Let $a\in J$ be orthogonal to the identity. Define $\mathbb{L}_a$ to be the operator given by $\tbtsmallmat{\alpha}{j}{j'}{\beta} \mapsto \tbtsmallmat{0}{a\cdot j}{-a\cdot j'}{0}$.  Note that $\mathbb{L}_a$ is a derivation of the Brown algebra by \cref{rem:Browne6,prop:derivationsstructure}.
\end{definition}
\begin{remark}\label{rem:compwise}
	We can write $\mathbb{L}_a$ as a block-diagonal matrix with respect to an obvious choice of basis. This will allow us to reuse a lot of computations from the $E_6$ case in \cite[Proposition~4.7]{desmetE6}.
\end{remark}
It turns out we can very easily find the full image of $\sigma$ using tools from representation theory.
\begin{proposition} Recall the notation from \cref{def:operators}. We have
	\[ \sigma(A(G)) = \spank_k\{ a_1\wedge a_2 \in \End(W) \mid a_1,a_2 \in W \}\cong \bigwedge\nolimits^2W\text{.} \]
\end{proposition}
\begin{proof}
	By computing characters, we know $\End(W) \cong k\oplus V(\omega_6) \oplus V(2\omega_7) \oplus V(\omega_1)$, and $\bigwedge^2W = k \oplus V(\omega_6)$. The map $\phi\colon \Sq^2W \to \End(W)\colon a\wedge b \mapsto a\wedge b$ is clearly bilinear and $G$-equivariant, and if $B$ is a basis, the set $\{ x\wedge y \in \End(W) \mid x,y\in B\}$ is linearly independent in $\End(W)$. This shows $\phi$ is a $G$-equivariant isomorphism, and thus $\spank_k\{ a_1\wedge a_2\mid a_1,a_2 \in W \} \cong \bigwedge\nolimits^2W$. Since $\sigma$ is $G$-equivariant, it maps $A(G)$ to an isomorphic representation, which then has to be equal to $\spank_k\{ a_1\wedge a_2 \in \End(W)\mid a_1,a_2 \in W \}$.
\end{proof}

We want to find the product for $A(G)$ in terms of the operators $a_1\wedge a_2\in \End(W)$. Thus we link the product $\diamond$ on $A(G)$ to the operators in question. To do this, we will use the explicit form of the derivations given in \cref{def:derE7}.
\begin{definition}
	Let $a,b\in \sigma(A(G))$. We define the multiplication $\star$ by \[a\star b \coloneqq \sigma(\sigma^{-1}(a)\diamond \sigma^{-1}(b)).\]
\end{definition}

\begin{proposition}\label{prop:examplecomp}
	Let $a,b\in \Oct$ be octonions such that $a\cdot a = b\cdot b =0$ and $\langle a,b\rangle \neq 0$. Then for $i\in \{1,2,3\}$ we have
	\[ \sigma(S(\mathbb{L}_{\upi{a}{i}}\mathbb{L}_{\upi{a}{i}})) = 54\tbtmat{0}{\upi{a}{i}}{0}{0} \wedge \tbtmat{0}{0}{\upi{a}{i}}{0}  , \]
	and
	\begin{multline*}
		\tbtmat{0}{\upi{a}{i}}{0}{0} \wedge \tbtmat{0}{0}{\upi{a}{i}}{0} \star \tbtmat{0}{\upi{b}{i}}{0}{0} \wedge \tbtmat{0}{0}{\upi{b}{i}}{0} \\
		= 
	\frac{1}{6}\langle a,b \rangle \tbtmat{0}{\upi{a}{i}}{0}{0} \wedge \tbtmat{0}{0}{\upi{b}{i}}{0}\\
	+ \frac{1}{6}\langle a,b \rangle \tbtmat{0}{\upi{b}{i}}{0}{0}\wedge \tbtmat{0}{0}{\upi{a}{i}}{0}\\
	 +\tfrac{1}{24} \langle a,b\rangle^2 (\mathbb{I}_j+\mathbb{I}_k)-\tfrac{1}{36} \langle a,b \rangle^2 \id_W.
	\end{multline*} 
\end{proposition}
\begin{proof}
	We can reuse the calculation in \cite[Proposition~4.7]{desmetE6}. It is clear that $\mathbb{L}_{\upi{a}{i}}^2$ acts as the zero map on the diagonal elements, and as $L_{\upi{a}{i}}^2$ on both subspaces $\tbtsmallmat{0}{J}{0}{0}$ and $\tbtsmallmat{0}{0}{J}{0}$. It was shown in \cite[Proposition~4.7]{desmetE6} that $L_{\upi{a}{i}}^2(x) = \tfrac{1}{2}\langle \upi{a}{i},x\rangle \upi{a}{i}$, and using this together with \cref{rem:sympformula}, we get the first equality.
	 For the second equality, we have 
	 \begin{multline}\label{eq:prop:examplecomp}
		\sigma(S(\mathbb{L}_{\upi{a}{i}}\mathbb{L}_{\upi{a}{i}})\diamond S(\mathbb{L}_{\upi{b}{i}}\mathbb{L}_{\upi{b}{i}}) )
		\\ = h^\vee \sigma(S([\mathbb{L}_{\upi{a}{i}},[\mathbb{L}_{\upi{a}{i}},\mathbb{L}_{\upi{b}{i}}]] \mathbb{L}_{\upi{b}{i}})) + h^\vee \sigma(S([\mathbb{L}_{\upi{b}{i}},[\mathbb{L}_{\upi{b}{i}},\mathbb{L}_{\upi{a}{i}}]] \mathbb{L}_{\upi{a}{i}}))\\
		+h^\vee \sigma(S([\mathbb{L}_{\upi{a}{i}},\mathbb{L}_{\upi{b}{i}}][\mathbb{L}_{\upi{a}{i}},\mathbb{L}_{\upi{b}{i}}])) + K(\mathbb{L}_{\upi{a}{i}},\mathbb{L}_{\upi{b}{i}}))S(\mathbb{L}_{\upi{a}{i}}\mathbb{L}_{\upi{b}{i}})).
	\end{multline}
	Composition of the derivations can be computed componentwise, e.g. $\mathbb{L}_{\upi{a}{i}}\circ \mathbb{L}_{\upi{b}{i}}\left(\tbtsmallmat{\alpha}{j}{\beta}{j^\prime}\right) = \tbtmat{0}{L_{\upi{a}{i}}\circ L_{\upi{b}{i}}(j)}{L_{\upi{a}{i}}\circ L_{\upi{b}{i}}(j^\prime)}{0}$. This way we get, using \cite[Lemma~2.8]{chayet2020class} and \cite[Proposition~4.7, Equation~(4.2,4.3)]{desmetE6}:
	\begin{multline*}
		{\sigma(S(\mathbb{L}_{\upi{a}{i}}\mathbb{L}_{\upi{a}{i}})\diamond S(\mathbb{L}_{\upi{b}{i}}\mathbb{L}_{\upi{b}{i}}) )}\vert_{\tbtsmallmat{0}{J}{0}{0}} \\ = 108\Big(-\tfrac{1}{2}h^\vee\langle \upi{a}{i},\upi{b}{i}\rangle  L_{\upi{a}{i}} \bullet L_{\upi{b}{i}} - \tfrac{1}{2} h^\vee\langle \upi{a}{i},\upi{b}{i}\rangle L_{\upi{a}{i}} \bullet L_{\upi{b}{i}}\\
		+h^\vee [L_{\upi{a}{i}},L_{\upi{b}{i}}]^2 + 18\langle \upi{a}{i},\upi{b}{i} \rangle L_{\upi{a}{i}}L_{\upi{b}{i}}\Big)
		\\ - 3\Big(-\tfrac{1}{2}h^\vee\langle a,b\rangle \Tr(L_{\upi{a}{i}} L_{\upi{b}{i}})) - \tfrac{1}{2} h^\vee\langle a,b\rangle\Tr(L_{\upi{a}{i}} L_{\upi{b}{i}}))\\
		+h^\vee \Tr([L_{\upi{a}{i}},L_{\upi{b}{i}}][L_{\upi{a}{i}},L_{\upi{b}{i}}])) + 18\langle \upi{a}{i},\upi{b}{i} \rangle \Tr(L_{\upi{a}{i}}L_{\upi{b}{i}}))\Big)
		\\ = 18\sigma(S([L_{\upi{a}{i}},L_{\upi{b}{i}}][L_{\upi{a}{i}},L_{\upi{b}{i}}])). 
	\end{multline*} 
	We can do the analogous computation for $\tbtsmallmat{0}{0}{J}{0}$ and the diagonal elements to get $\sigma(S(\mathbb{L}_{\upi{a}{i}}\mathbb{L}_{\upi{a}{i}})\diamond S(\mathbb{L}_{\upi{b}{i}}\mathbb{L}_{\upi{b}{i}}) ) = 18 \sigma(S([\mathbb{L}_{\upi{a}{i}},\mathbb{L}_{\upi{b}{i}}][\mathbb{L}_{\upi{a}{i}},\mathbb{L}_{\upi{b}{i}}]))$. Using \cite[Lemma~5.2(iii))]{jariprevious}, we deduce the equality in the statement of the proposition.
\end{proof}

\subsection{Defining multiplications on $V$}\label{sec:defmultg2}
Recall from \cref{modulestructure} that $A(G) \cong k\oplus V$, where $V$ is the irreducible $1539$-dimensional representation of the group $G$ of type $E_7$. By \cite[Example~A.6]{chayet2020class}, the multiplication is of the form
\begin{equation}\label{eq:multag2}
	(\lambda, u)\star (\mu,v) = (\lambda\mu+f(u,v), \lambda v+\mu u + u\odot v)\text{,} 
\end{equation}
for a certain $G$-invariant symmetric bilinear form $f$ and $G$-equivariant symmetric multiplication $\odot$ on $V$. 
\begin{lemma}\label{lem:reptheory}
	There is only a one-dimensional space of symmetric invariant bilinear forms on $V$ and only a two-dimensional space of symmetric multiplications on $V$ that are $G$-equivariant.	
\end{lemma}
\begin{proof}
	By computing characters (\cref{reptheoryremark}), we can see that the multiplicity of $k$ in $\Sq^2V$ is $1$ and and the multiplicity of $V$ in $\Sq^2V$ is $2$. The lemma now follows from Schur's lemma. 
\end{proof}

Note that the embedding $\sigma$ sends $V$ to the subspace of trace zero elements, i.e.\@
\[\sigma(V) = \left\{ \sum_{i} a_i\wedge b_i \,\middle\vert\, \sum_{i} \omega(a_i,b_i)=0 \right\}\text{.} \]

We also know that the bilinear form in \cref{eq:multag2} will have to be a scalar multiple of the following form, restricted to $\sigma(V)$.
\begin{definition}\label{def:fprime}
	Define $f'\colon \bigwedge\nolimits^2W \times \bigwedge\nolimits^2W \to k$ by the formula \[f'(a\wedge b,c\wedge d) = \omega(a,c)\omega(b,d) - \omega(a,d)\omega(b,c) . \]
\end{definition}

To define commutative $G$-equivariant multiplications $\odot$ on $V\subseteq\sigma(A(G))$, we will first make an observation about symmetric operators.
\begin{definition}
	We will call an operator $f\in \End(W)$ \emph{symmetric} (respectively \emph{antisymmetric}) with respect to $\omega$ if $\omega (f(x),y) = \omega(x,f(y))$ (respectively $\omega (f(x),y) = -\omega(x,f(y))$) for all $x,y\in W$.
\end{definition}
Note that with this definition $\sigma(A(G))$ is precisely the space of symmetric operators with respect to $\omega$ on $W$. Remark also that the \emph{Jordan product}, denoted $f\bullet g = \tfrac{1}{2}(f\circ g+g\circ f)$, of two antisymmetric operators is always itself symmetric.
\begin{lemma}\label{lem:jordan}
	Let $a,b,c,d\in W$. Then $L_{a,b}\bullet L_{c,d}$ is a symmetric operator with respect to $\omega$.
\end{lemma}
\begin{proof}
	 Note that the expression $\omega(x,t(y,z,w))$ is invariant under rearrangement of the arguments by the definition of a Freudenthal triple system (\cite[Definition 3.1]{garibaldie6e7}). In particular we have $\omega(x,t(y,z,w)) = -\omega(t(y,z,x),w)$ Thus we have for all $x,y \in W$ that $\omega(t(a,b,t(c,d,x)),y) = -\omega(t(c,d,x),t(a,b,y)) = \omega(x,t(c,d,t(a,b,y)))$ The result now easily follows.
\end{proof}

Using the above considerations, we now turn to finding a basis for the space of $G$-invariant multiplications. It suffices to find two linearly independent multiplications by \cref{lem:reptheory}.
\subsubsection{Finding multiplications}	 A first multiplication is easily found by using the Jordan product.
\begin{proposition}\label{prop:mult1}
Define for $a,b,c,d\in W $
\begin{multline}\label{eq:def:odot1}
	a\wedge b \odot_1 c\wedge d  \\
	\coloneqq
	\omega(a,c)b\wedge d +\omega(b,d)a\wedge c - \omega(a,d)b\wedge c -\omega(b,c)a\wedge d - 4f'(a\wedge b ,c\wedge d)\text{.}  
\end{multline}
Then $\odot_1$ is a commutative $G$-equivariant product on $\sigma(A(G))$ and restricts to a well-defined commutative $G$-equivariant product on $\sigma(V)$.
\end{proposition}	
\begin{proof}
	The product $\odot_1$ is clearly commutative and $G$-equivariant, and one can check that for $\sum_i a_i\wedge b_i,  \sum_i c_i\wedge d_i\in \sigma(V)$ we have $\sum_i a_i\wedge b_i\odot_1  \sum_i c_i\wedge d_i\in \sigma(V)$.
\end{proof}
We can find a second multiplication using \cref{lem:jordan}.
\begin{proposition}\label{prop:mult2}
	Define for $a,b,c,d\in W $
\begin{multline}\label{eq:def:odot2}
	a\wedge b \odot_2  c\wedge d  
	\coloneqq
	L_{a,c}\bullet L_{b,d}-L_{a,d}\bullet L_{b,c} +\tfrac{11}{28}f'(a\wedge b, c\wedge d)\id_W \\ = \sum_{x_i,y_i} -\tfrac{1}{4}\Big(t(a,c,x_i) \wedge t(b,d,y_i) + t(b,d,x_i) \wedge t(a,c,y_i) - t(a,d,x_i)\wedge t(b,c,y_i) - t(b,c,x_i)\wedge t(a,d,y_i)\Big) \\ +\tfrac{11}{28}f'(a\wedge b, c\wedge d)\id_W\text{,}  
\end{multline}
where the summation runs over dual bases $\{x_i\},\{y_i\}$ with respect to $\omega$.
Then $\odot_2$ is a commutative $G$-equivariant product on $\sigma(A(G))$ and restricts to a well-defined commutative $G$-equivariant product on $\sigma(V)$.
\end{proposition}
\begin{proof}
	First, to realise why the two expressions used to define the product are equal, note that $\id_W = \tfrac{1}{2}\sum x_i \wedge y_i$.  For $a,b,c,d,x,y,z  \in W $, we have \begin{multline*}
		(L_{a,b}\circ x\wedge y  \circ L_{c,d} + L_{c,d}\circ x\wedge y  \circ L_{a,b} )(z)  \\
		= \omega(x,t(c,d,z))t(a,b,y) - \omega(y, t(c,d,z) )t(a,b,x)+\omega(x,t(a,b,z))t(c,d,y) - \omega(y, t(a,b,z) )t(c,d,x) \\
		=-\omega(t(c,d,x),z)t(a,b,y) + \omega(t(c,d,y) ,z )t(a,b,x)-\omega(t(a,b,x),z)t(c,d,y) + \omega(t(a,b,y), z )t(c,d,x) \\
		= -(t(a,b,x)\wedge t(c,d,y) + t(c,d,x)\wedge t(a,b,y) )(z).
	\end{multline*} 
	Combining this with the expression $\id_W=\tfrac{1}{2}\sum_i x_i \wedge y_i$, we get  $L_{a,b}\bullet L_{c,d} = \tfrac{1}{2}(L_{a,b}\circ \id_W \circ L_{c,d} + L_{c,d}\circ \id_W \circ L_{a,b} ) = -\tfrac{1}{4}\sum_i t(a,b,x_i) \wedge t(c,d,y_i) + t(c,d,x_i) \wedge t(a,b,y_i)$. This proves the two expressions defining $\odot_2$ are equal.
	
	Clearly $\odot_2$ is commutative and $G$-equivariant. To compute $a\wedge b \odot_2  c\wedge d$ for certain $a,b,c,d\in \sigma(A(G))$, we need to compute $L_{a,b}\bullet L_{c,d}(x) = \tfrac{1}{2}(t(a,b,t(c,d,x)) + t(c,d,t(a,b,x))$ for an arbitrary $x\in B$. This is very cumbersome to do by hand, and takes a long time, due to the fact that $t$ does not have as nice an expression as $\omega$ (\cref{rem:sympformula}). We do not reproduce the calculations here for brevity, but do provide a Jupyter notebook with calculations \cite{desmet_computations}. This gives
	\begin{multline}\label{eq:exodot2}
			\tbtmat{0}{\upi{a}{i}}{0}{0} \wedge \tbtmat{0}{0}{\upi{a}{i}}{0} \odot_2 \tbtmat{0}{\upi{b}{i}}{0}{0} \wedge \tbtmat{0}{0}{\upi{b}{i}}{0}  \\ = 4\langle a, b \rangle^2 (\mathbb{I}_j+\mathbb{I}_k) +\frac{9}{2}\langle a,b\rangle \tbtmat{0}{\upi{a}{i}}{0}{0} \wedge \tbtmat{0}{0}{\upi{b}{i}}{0} \\
	+ \frac{9}{2}\langle a,b \rangle \tbtmat{0}{\upi{b}{i}}{0}{0}\wedge  \tbtmat{0}{0}{\upi{a}{i}}{0} - \tfrac{73}{28}\langle a,b\rangle^2\id_W\text{.}\end{multline}
We still need to check $\odot_2|_{\sigma(V)}$ has image in $\sigma(V)$. To do so, note that for $a,b \in \sigma(V)$, the element $a\odot_2 b$ is contained in $\sigma(V)$ if and only if $\Tr(a\odot_2 b) = 0$. Now, the map $S^2V \to k \colon ab \mapsto \Tr(a\odot_2 b)$ is $G$-equivariant as the composition of two $G$-equivariant maps, and thus it is equal to $\lambda f'$ for some scalar $\lambda\in k$. 
One can check that the trace of the right hand side of \cref{eq:exodot2} is zero, but $f'\left(\tbtsmallmat{0}{\upi{a}{i}}{0}{0} \wedge \tbtsmallmat{0}{0}{\upi{a}{i}}{0},\tbtsmallmat{0}{\upi{b}{i}}{0}{0} \wedge \tbtsmallmat{0}{0}{\upi{b}{i}}{0} \right)$ is non-zero. This implies $\lambda = 0$, and thus that $\odot_2$ has image in $\sigma(V)$ when restricted to $\sigma(V)$.
\end{proof}
\begin{proposition}\label{prop:multbasis}
	The space of commutative $G$-equivariant products on the representation $V$ of highest weight $\omega_6$ is spanned by $\odot_1$ and $\odot_2$, restricted to $\sigma(V)$, defined in \cref{prop:mult1,prop:mult2}.
\end{proposition}
\begin{proof}
	The fact that these are both commutative $G$-equivariant products is proven in \cref{prop:mult1,prop:mult2}. The fact that they are linearly independent follows from \cref{eq:exodot2} and the computation below:
	\begin{multline}\label{eq:exodot1}
		\tbtmat{0}{\upi{a}{i}}{0}{0} \wedge \tbtmat{0}{0}{\upi{a}{i}}{0} \odot_1 \tbtmat{0}{\upi{b}{i}}{0}{0} \wedge \tbtmat{0}{0}{\upi{b}{i}}{0}  \\ 
		= -\omega\left(\tbtmat{0}{0}{\upi{a}{i}}{0} , \tbtmat{0}{\upi{b}{i}}{0}{0} \right)\tbtmat{0}{\upi{a}{i}}{0}{0} \wedge \tbtmat{0}{0}{\upi{b}{i}}{0} \\
	- \omega\left(\tbtmat{0}{\upi{a}{i}}{0}{0} , \tbtmat{0}{0}{\upi{b}{i}}{0}\right)\tbtmat{0}{0}{\upi{a}{i}}{0}\wedge \tbtmat{0}{\upi{b}{i}}{0}{0} -4\langle a,b\rangle^2\id_W\\
	= \langle a,b\rangle \left(\tbtmat{0}{\upi{a}{i}}{0}{0} \wedge \tbtmat{0}{0}{\upi{b}{i}}{0} 
	+ \tbtmat{0}{\upi{b}{i}}{0}{0}  \wedge \tbtmat{0}{0}{\upi{a}{i}}{0}\right)-4\langle a,b\rangle^2\id_W.
	\end{multline}
	As the space of commutative $G$-equivariant products on $V(\omega_6)$ is $2$-dimensional (\cref{reptheoryremark}), the result follows.
\end{proof}

\subsection{Calculating parameters}

We have finally gathered enough information to prove the main theorem of this section.
\begin{theorem}\label{identificationag2}
	Suppose $G$ is a group of type $E_7$ with trivial Tits algebras over $k$, and $W$ its corresponding Brown algebra. The algebra $A(G)$ is isomorphic to the subspace $\spank_k\{a\wedge b\mid a,b \in W\}\subset \End(W)$ endowed with the multiplication $\star$ given by
	\begin{multline}\label{eq:multmaintheorem}
		a\wedge b \star c\wedge d  = \tfrac{1}{96}(L_{a,c}\bullet L_{b,d}-L_{a,d}\bullet L_{b,c}) \\
		+\tfrac{23}{192}(\omega(a,c)b\wedge d +\omega(b,d)a\wedge c - \omega(a,d)b\wedge c -\omega(b,c)a\wedge d)\\
		+\tfrac{1}{32}(\omega(a,b)c\wedge d  +\omega(c,d)a\wedge b) +\tfrac{1}{288}(2\omega(a,b)\omega(c,d)+f'(a\wedge b,c\wedge d))\id_W \text{,}
	\end{multline}
	where $f'$ is as in \cref{def:fprime}.	
	
\end{theorem}
\begin{proof}
	We may assume $k$ is algebraically closed. Note that the space of commutative $G$-equivariant multiplications on $k\oplus V$ is $5$-dimensional, since $\Sq^2(k\oplus V) \cong k\oplus V \oplus \Sq^2(V)$ contains $k$ with multiplicity $2$, and $V$ with multiplicity $3$. This implies that the multiplication has to be of the form
	\begin{equation*}
		a\star b = \rho_1a\odot_1 b + \rho_2a\odot_2 b + \lambda (\Tr(a)b + \Tr(b)a) + (\rho_3\Tr(a)\Tr(b) + \mu\Tr(a\circ b))\id_W,
 	\end{equation*}
	for all  $a,b\in \sigma(A(G)) \subset\End(W)$, with $\rho_1,\rho_2,\rho_3, \mu,\lambda\in k$.

	   By comparing \cref{eq:exodot2,eq:exodot1,prop:examplecomp}, we obtain for $a\wedge b , c\wedge d \in A(G)$ that 
	\begin{multline}\label{eq:idunit}
		a\wedge b \star c\wedge d  = \tfrac{1}{96}(L_{a,c}\bullet L_{b,d}-L_{a,d}\bullet L_{b,c}) \\
		+\tfrac{23}{192}(\omega(a,c)b\wedge d +\omega(b,d)a\wedge c - \omega(a,d)b\wedge c -\omega(b,c)a\wedge d)\\
		+\lambda(\omega(a,b)c\wedge d  +\omega(c,d)a\wedge b) +\tfrac{1}{288}(\mu\omega(a,b)\omega(c,d)+f'(a\wedge b,c\wedge d))\id_W,
	\end{multline}
	for certain $\lambda,\mu \in k$. It remains to use the fact that the identity operator is the unit of the $\star$ operation, to determine $\lambda$ and $\mu$. Indeed, for arbitrary $a,b \in W $ we have $
		a\wedge b \star \id_W = (28\lambda + \nu_1) a\wedge b +(\tfrac{7}{72}\mu+\nu_2)\id_W$ for certain $\nu_1,\nu_2$ that can be computed from \cref{eq:idunit}. Since $\id_W$ has to be the unit of the multiplication $\star$, we should have $28\lambda + \nu_1 =1 $ and $\tfrac{7}{72}\mu+\nu_2=0$. Computing this is straightforward, yet too tedious to do by hand, we used a computer (\cite{desmet_computations}) to compute the scalars $\nu_1,\nu_2$ in \cref{eq:idunit}. Then one finds $\lambda = \tfrac{1-\nu_1}{28} = \frac{1}{32}$ and $\mu = -\tfrac{72}{7}\nu_2 = 2 $.  	
\end{proof}
\section{The automorphism group}\label{autgroup}
Now that we have an explicit construction for the algebra $A(G)$ we can verify that its automorphism group is precisely $G$ in the same way we did in \cite[Section 4]{jariprevious}. Briefly, if $G$ and $G'$ are two groups of type $E_7$ in $\Sp_{56}$, they are conjugate, by a theorem of Frobenius \cite[Theorem 1, p.\@ 14]{Malcevconjugacy}. From this and the fact that $G$ has only inner automorphisms, we can deduce that there is exactly one copy $B'$ of $\Sp_{56}$ such that $G\subset B' \subset \SO(V,f')$. It remains to check that the algebra product is not Lie invariant under $\Lie(B')$. 
\begin{proposition}\label{prop:lieinvariance}
	The product $\odot_2$ is not Lie invariant under $\Lie(B')$. As a consequence, the product $\odot_2$ is not invariant under $B'$.
\end{proposition}
\begin{proof}
	Let $j\neq i \in \{1,2,3\}$ and $a,b$ two octonions with $\langle a,a\rangle = \langle b,b \rangle = 0$ but $\langle a,b\rangle \neq 0$.  Let $D\colon W\to W$ denote the operator switching $ \tbtsmallmat{1}{0}{0}{0}$ with $\tbtsmallmat{0}{\upi{e}{j}}{0}{0}$, and $\tbtsmallmat{0}{0}{0}{1}$ with $ \tbtsmallmat{0}{0}{\upi{e}{j}}{0}$, and sending everything orthogonal (with respect to $\omega$) to the subspace spanned by these four elements to zero. Then $\omega(D(x),y) + \omega(x,D(y)) =0$ for all $x,y\in A(G)$, thus $D$ is a derivation of the symplectic form $\omega$,
	so $D\in \Lie(B')$. Its action on $\bigwedge^2 W$ is determined as follows: $D(a\wedge b) \coloneqq (Da)\wedge b + a\wedge (Db)$. Let $x$ denote the right-hand side of \cref{eq:exodot2}. For $\odot_2$ to be invariant under $\Lie(B')$, we should have that $D(x) = 0 \odot_2 \tbtsmallmat{0}{\upi{b}{i}}{0}{0} \wedge \tbtsmallmat{0}{0}{\upi{b}{i}}{0} + \tbtsmallmat{0}{\upi{a}{i}}{0}{0} \wedge \tbtsmallmat{0}{0}{\upi{a}{i}}{0}\odot_2 0 =0$. However, $D(x)$ is not zero, since the coefficients of $\tbtsmallmat{1}{0}{0}{0}\wedge \tbtsmallmat{0}{0}{0}{1}$ and $\tbtsmallmat{0}{\upi{e}{j}}{0}{0}\wedge \tbtsmallmat{0}{0}{\upi{e}{j}}{0}$ in $x$ do not agree (with respect to an obvious choice of basis).
\end{proof}
\begin{corollary}\label{cor:autgroup}
	The automorphism group of $A(G)$ is precisely $G/Z(G)$, an adjoint group of type $E_7$.
\end{corollary}
\begin{proof}
	By the argument of \cite[Lemma 5.5]{GG15} (see also \cite[Corollary 4.7 and Theorem 4.8]{jariprevious}), either $G/Z(G)$ is the identity component of the automorphism group of $\odot$ (see \cref{eq:multag2}), or there is a group $H$ of type $C_{28}$ that contains $G/Z(G)$ such that $\odot$ and $f'$ are Lie invariant under $\Lie(H)$. But a group of type $C_{28}$ that stabilizes $f'$ and contains $G/Z(G)$ has to be $B'$. By \cref{prop:lieinvariance} then, $\odot$ is not Lie invariant under $\Lie(B')$. We conclude that $G/Z(G)$ has to be the identity component of the automorphism group of $\odot$, and thus also of $\star$. As a consequence, since $E_7$ has no outer automorphisms, $G/Z(G)$ is the entire automorphism group (see also \cite[Theorem 4.8]{jariprevious}).
\end{proof}
\section{Arbitrary groups of type $E_7$}\label{sec:e7arbitrary}
In \cite{garibaldie7arbitrary}, Skip Garibaldi used \emph{generalised Freudenthal triple systems} to describe arbitrary groups of type $E_7$. We will use these objects to extend our description of $A(G)$ to also include arbitrary adjoint groups of type $E_7$. To describe these objects, we need some concepts from the theory of central simple algebras. We will not go into much detail here, but everything we use is introduced in \cite[\S1-2]{involutions}.

Recall that a \emph{central simple algebra} $A$ over $k$ is an algebra which becomes isomorphic to a matrix algebra $M_n\left(\bar{k}\right)$ when base changed to an algebraic closure $\bar{k}$ of $k$. In this case, we call $n$ the \emph{degree} of the central simple algebra. An \emph{involution} on $A$ is an anti-automorphism of order $2$. We call $\sigma$ \emph{symplectic} if $\sigma_{\bar{k}}$ corresponds to a skew-symmetric bilinear form as in \cite[p.1]{involutions}.


 The \emph{sandwich map}
\[\Sand \colon A \otimes_k A \to  \End_k(A)\]
defined by
$\Sand(a \otimes b)(x) = axb $ for $a, b, x \in A$
is an isomorphism of $k$-vector spaces by \cite[3.4]{involutions}. As in \cite[\S 8.B]{involutions}, we
have a map $\sigma_2$ on $A \otimes_k A$ which is defined by the equation
\[\Sand(\sigma_2(u))(x) = \Sand(u)(\sigma(x)) \text{ for } u \in  A \otimes_k A, x \in  A.\]
If $A$ is split, then $A \cong \End_k (V)$ for some $k$-vector space $V$, and $\sigma$ is
the adjoint involution for some nondegenerate skew-symmetric bilinear form $\omega$ on $V$ (i.e.,
$\omega(f (x), y) = \omega(x, \sigma(f)(y)) $for all $f \in \End_k (V )$). There is an identification $\phi_\omega :
V \otimes_k V 
\to \End_k (V )$ given by $\phi_\omega(	a\otimes b) = a\omega(b,\mathunderscore)$. Note that  $a\wedge b = \phi_\omega(b\otimes a -	a\otimes b) $ in our notation. By \cite[8.6]{involutions}, $\sigma_2$ is
given by
\[\sigma_2(\phi_\omega(a \otimes b) \otimes \phi_\omega(c \otimes d)) = -\phi_\omega(a \otimes c) \otimes \phi_\omega(b \otimes d)\]
for $a,b,c,d\in V$.

Lastly, for any map $f\colon A\to A$ we define the map $\hat{f} \colon A\otimes_k A\to A$ by $\hat f = m\circ (f\otimes \id_A)$, i.e.\@ $\hat {f} (a\otimes b) = f(a)b$. We are now ready to give the definition of a generalised Freudenthal triple system. 

\begin{definition}[{\cite[Definition~3.2]{garibaldie7arbitrary}}] \label{def:gift}
A generalised Freudenthal triple system or \emph{gift} over a field $k$ is a triple $\bgift$ such that $A$ is a
central simple $k$-algebra of degree $56$, $\sigma$~is a symplectic
involution on $A$, and $\pi \colon A \to A$ is a $k$-linear
map such that
\begin{enumerate}[label = (G\arabic*)]
	\item $\sigma \pi(a) = \pi \sigma(a) = -\pi(a)$,
	\item $a \pi(a) \neq 2a^2$ for some $a \in \Skew(A,\sigma)$, 
	\item $\pi(\pi(a) a) = 0$ for all $a \in \Skew(A,\sigma)$, 
	\item $\hat{\pi} - \hat{\sigma} - \hat{\id} = - (\hat{\pi} - \hat\sigma - \hat\id) \sigma_2$,
	\item $\Trd(\pi(a)\pi(a')) = -24 \Trd(\pi(a) a')$ for all $a, a'
  \in (A,\sigma)$. 
\end{enumerate}

\end{definition}

It turns out that nondegenerate Freudenthal triple systems (i.e.\@ Brown algebras) give examples of gifts (hence the name). More precisely, let $W$ be a Brown algebra. If we define $p\colon W\otimes_k W  \to \End_k(W)$ by $p(a\otimes b)(x) = t(a,b,x)+ \omega(a,x)b+\omega(b,x)a $, then $(\End_k(W),\sigma_{\omega},p\phi_\omega^{-1})$ is a gift by \cite[Example~3.3]{garibaldie7arbitrary}. In fact, we have even more:

\begin{lemma}[{\cite[Lemma~3.6]{garibaldie7arbitrary}}]\label{lem:splitgift}
	 Suppose $\bgift$ is a gift. The central simple algebra $A$ is split if and only if $\bgift \cong (\End(W),\sigma_\omega, p\phi_\omega^{-1})$ for some Brown algebra $W$.
\end{lemma}

The operators $\pi,\sigma_2$ defined above allow us to describe $A(G)$ for arbitrary groups $G$ of type $E_7$. This relies on the fact that gifts are equivalent to groups of type $E_7$, see \cite[Theorem~3.13]{garibaldie7arbitrary}.

\begin{proposition}
	Let $\bgift$ be a gift over $k$. Let $C = (\Symm(A,\sigma),\circledast )$ be the  algebra with multiplication $\circledast$ given by
	\begin{multline}\label{eq:multgifts}
		a\circledast b \coloneqq \tfrac{1}{4} a\bullet b - \tfrac{1}{192} m \circ (\pi \otimes \pi) \circ \sigma_2(a\otimes b) + \tfrac{1}{48}(\Trd(a)b + \Trd(b)a) \\+ \tfrac{1}{576}(\Trd(a)\Trd(b) +\Trd(ab))1_A,
	\end{multline} 
	where $m\colon A\otimes_k A \to A$ is the multiplication map of the central simple algebra, and $\bullet$ denotes the usual Jordan product. Then $C$ is a twisted form of $A(E_7^{\ad})$, where $E_7^{\ad}$ denotes the split adjoint group of type $E_7$ over $k$.
\end{proposition}
\begin{proof}	
	Suppose $\bgift$ is split. Then it is isomorphic to $(\End_k(W),\sigma_\omega,p\phi_\omega^-1) $, for a non-degenerate Freudenthal triple system $(W,\omega,t)$ by \cref{lem:splitgift}. In this case, we want to compute $a\wedge b\circledast c\wedge d$ and show it is equal to $a\wedge b\star c\wedge d$ as in \cref{identificationag2}. We get
	\begin{align}\label{eq:bullet}
		a\wedge b \bullet c\wedge d &= \tfrac{1}{2}(\omega(a,c)b\wedge d +\omega(b,d)a\wedge c - \omega(a,d)b\wedge c -\omega(b,c)a\wedge d) ),
	\end{align}
	and 
	\begin{align*}
		&m \circ (\pi \otimes \pi) \circ \sigma_2(a\wedge b \otimes c\wedge d) \\
		&= -p(a\otimes c)p(b\otimes d) -p(b\otimes d)p(a\otimes c)+p(a\otimes d)p(b\otimes c) +p(b\otimes c)p(a\otimes d) \\
		&=-(L_{a,c}+\phi_\omega(a\otimes c + c\otimes a))(L_{b,d}+\phi_\omega(b\otimes d + d\otimes b))  \\
		&\hspace{2ex}- (L_{b,d}+\phi_\omega(b\otimes d +d\otimes b))(L_{a,c}+\phi_\omega(a\otimes c + c\otimes a))\\
		&\hspace{2ex}+(L_{a,d}+\phi_\omega(a\otimes d + d\otimes a))(L_{b,c}+\phi_\omega(b\otimes c + c\otimes b)) \\
		&\hspace{2ex}+(L_{b,c}+\phi_\omega(b\otimes c + c\otimes b))(L_{a,d}+\phi_\omega(a\otimes d + d\otimes a)).
	\end{align*}
	Expanding this gives
	\begin{multline}\label{eq:sandwicheqjordan}
	m \circ (\pi \otimes \pi) \circ \sigma_2(a\wedge b \otimes c\wedge d) 
		= 2(L_{a,d}\bullet L_{b,c}-L_{a,c}\bullet L_{b,d}) \\ 
		-\phi_\omega(a\otimes c + c\otimes a)L_{b,d} + \phi_\omega(a\otimes d + d\otimes a)L_{b,c} -\phi_\omega(b\otimes d + d\otimes b)L_{a,c} +\phi_\omega(b\otimes c + c\otimes b)L_{a,d}\\ 
		-L_{b,d}\phi_\omega(a\otimes c + c\otimes a) + L_{b,c}\phi_\omega(a\otimes d + d\otimes a) -L_{a,c}\phi_\omega(b\otimes d + d\otimes b) +L_{a,d}\phi_\omega(b\otimes c + c\otimes b)\\
		+ 2\omega(c,d)a\wedge b + 2\omega(a,b)c\wedge d + \omega(a,c) b\wedge d +\omega(b,d) a\wedge c -\omega (a,d) b\wedge c -  \omega(b,c)a\wedge d.
	\end{multline}
	
	Now note that for $x,y,u,v\in W$ we have $\phi_\omega(x\otimes y) L_{u,v} = -\phi_\omega(x\otimes t(u,v,y))$ and $ L_{u,v} \phi_\omega(x\otimes y) = \phi_\omega(t(u,v,x)\otimes y)$. These two equations imply that the second and third row of \cref{eq:sandwicheqjordan} summed together becomes zero, hence 
	\begin{multline}\label{eq:sandwich}
		m \circ (\pi \otimes \pi) \circ \sigma_2(a\wedge b \otimes c\wedge d) \\ 
		= 2(L_{a,d}\bullet L_{b,c}-L_{a,c}\bullet L_{b,d} )
		+ \omega(a,c) b\wedge d +\omega(b,d) a\wedge c -\omega (a,d) b\wedge c -  \omega(b,c)a\wedge d\\
		+ 2\omega(c,d)a\wedge b + 2\omega(a,b)c\wedge d.
	\end{multline}
	Combining \cref{eq:bullet,eq:sandwich}, we get
	\begin{multline*}
		a\wedge b\circledast c\wedge d = \tfrac{1}{4} a\wedge b\bullet c\wedge d - \tfrac{1}{192} m \circ (\pi \otimes \pi) \circ \sigma_2(a\wedge b\otimes c\wedge d) \\+ \tfrac{1}{48}(\Trd(a\wedge b)c\wedge d + \Trd(c\wedge d)a\wedge b) \\+ \tfrac{1}{576}(\Tr(a\wedge b)\Tr(c\wedge d) +\Tr((a\wedge b)(c\wedge d))\id_W\\
		=\tfrac{1}{96}(L_{a,c}\bullet L_{b,d}-L_{a,d}\bullet L_{b,c}) 
		+\tfrac{23}{192}(\omega(a,c)b\wedge d +\omega(b,d)a\wedge c - \omega(a,d)b\wedge c -\omega(b,c)a\wedge d)\\
		+\tfrac{1}{32}(\omega(a,b)c\wedge d  +\omega(c,d)a\wedge b) +\tfrac{1}{288}(2\omega(a,b)\omega(c,d)+f'(a\wedge b,c\wedge d))\id_W \text{.}
	\end{multline*}
	This shows that for any gift $\bgift$, the algebra $(\Symm(A),\circledast)$ is a twisted form of $A(E_7^{\mathrm{ad}})$, since any central simple algebra splits over a finite field extension.
\end{proof}
In fact, since the automorphism group of $A(E_7^{\mathrm{ad}})$ is equal to $E_7^{\mathrm{ad}}$, twisted forms of $A(E_7^{\mathrm{ad}})$ are parametrised by $H^1(k,E_7^{\mathrm{ad}})$, and so are the algebras $A(G)$. It only makes sense there is a natural correspondence, but we will prove this using a technical condition known as the \emph{descent condition} \cite[p.\@ 361]{involutions}. The strategy is completely analogous to \cite[Theorem~3.13]{garibaldie7arbitrary} and \cite[26.2--26.6,26.9,26.12,26.14,26.15,26.18,26.19]{involutions}. The terms \emph{$\Gamma$-embedding}, \emph{$\Gamma$-extension} and \emph{descent condition} are all defined in \cite[p.\@ 360-361]{involutions}
\begin{theorem}
	Let $G$ be an arbitrary group of type $E_7$ over a field $k$ with $\kar k =0$ or $\kar k >19$. Let $\bgift$ be its associated gift. Then the algebra $A(G)$ is isomorphic to $(\Symm(A,\sigma),\circledast )$ with multiplication given by
	\begin{multline}\label{eq:multgifts}
		a\circledast b \coloneqq \tfrac{1}{4} a\bullet b - \tfrac{1}{192} m \circ (\pi \otimes \pi) \circ \sigma_2(a\otimes b) + \tfrac{1}{48}(\Trd(a)b + \Trd(b)a) \\+ \tfrac{1}{576}(\Trd(a)\Trd(b) +\Trd(ab))1_A,
	\end{multline} 
	where $m\colon A\otimes_k A \to A$ is the multiplication map of the central simple algebra, and $\bullet$ denotes the usual Jordan product.
\end{theorem}
\begin{proof}
	Let $\mathbf{CG}(k)$ denote the category of twisted forms of $A(E_7^{\mathrm{ad}})$ over $k$ with morphisms being algebra isomorphisms, and let $\mathbf{E_7^{\mathrm{ad}}}(k)$ denote the category of twisted forms of $ E_7^{\mathrm{ad}}$ over $k$, with morphisms being isomorphisms of algebraic groups. Then $i\colon \mathbf{CG}(k) \to \mathbf{CG}(k_{\mathrm{sep}})\colon C \rightsquigarrow C\otimes k_{\mathrm{sep}}$ and $j\colon \mathbf{E_7^{\mathrm{ad}}}(k) \to \mathbf{CG}(k_{\mathrm{sep}})\colon G \rightsquigarrow G_{k_{\mathrm{sep}}}$ are \emph{$\Gamma$-embeddings}, where $\Gamma = \mathrm{Gal}(k_{\mathrm{sep}}/k)$ is the absolute Galois group \cite[Examples~(26.1)]{involutions}. There is a functor $\mathbf{S}\colon \mathbf{CG}(k) \to \mathbf{E_7^{\mathrm{ad}}}(k) \colon C \rightsquigarrow \Aut(C)$, with a \emph{$\Gamma$-extension} $\widetilde{\mathbf{S}}\colon \mathbf{CG}(k_{\mathrm{sep}}) \to \mathbf{E_7^{\mathrm{ad}}}(k_{\mathrm{sep}}) \colon C \rightsquigarrow \Aut(C)$. Completely analogous to \cite[Corollary~26.6]{involutions}, $i$ satisfies the descent condition. The functor $\widetilde{\mathbf{S}}$ is clearly an equivalence of categories, since it is fully faithful and essentially surjective. This implies by \cite[Proposition~26.2]{involutions} that $\mathbf{S}$ is also an equivalence of categories. In particular, two twisted forms of $A(E_7^{\mathrm{ad}})$ with isomorphic automorphism group schemes are in fact isomorphic.

	We may suppose without loss of generality that $G$ is adjoint, by passing to the quotient by its centre as $A(G)= A(G/Z(G))$. Let $\bgift$ be a gift with $\Aut\bgift \cong G$. This exists by \cite[Theorem~3.13]{garibaldie7arbitrary}. There is an obvious containment $\Aut\bgift \subseteq \Aut(\Symm(A),\circledast)$, and since these are both connected and smooth with the same dimension, they are equal. This shows $ \Aut(\Symm(A),\circledast)\cong G$, and by the previous paragraph $(\Symm(A),\circledast)\cong A(G)$.
	\end{proof}

\bibliographystyle{alpha}
\bibliography{sources.bib}

\newcommand{\etalchar}[1]{$^{#1}$}
\begin{thebibliography}{KMRT98}

\bibitem[AF84]{cayleydickson}
B.~N. Allison and J.~R. Faulkner.
\newblock A {C}ayley-{D}ickson process for a class of structurable algebras.
\newblock {\em Trans. Amer. Math. Soc.}, 283(1):185--210, 1984.

\bibitem[All78]{defstructurablealg}
B.~N. Allison.
\newblock A class of nonassociative algebras with involution containing the
  class of {J}ordan algebras.
\newblock {\em Math. Ann.}, 237(2):133--156, 1978.

\bibitem[BDM13]{tomexceptionalq}
Lien Boelaert and Tom De~Medts.
\newblock Exceptional {M}oufang quadrangles and structurable algebras.
\newblock {\em Proc. Lond. Math. Soc. (3)}, 107(3):590--626, 2013.

\bibitem[Bou02]{Bourbaki46}
Nicolas Bourbaki.
\newblock {\em Lie groups and {L}ie algebras. {C}hapters 4--6}.
\newblock Elements of Mathematics (Berlin). Springer-Verlag, Berlin, 2002.
\newblock Translated from the 1968 French original by Andrew Pressley.

\bibitem[Bro63]{defbrown}
Robert~B. Brown.
\newblock A new type of nonassociative algebras.
\newblock {\em Proc. Nat. Acad. Sci. U.S.A.}, 50:947--949, 1963.

\bibitem[CG21]{chayet2020class}
Maurice Chayet and Skip Garibaldi.
\newblock A class of continuous non-associative algebras arising from algebraic
  groups including {$E_8$}.
\newblock {\em Forum Math. Sigma}, 9:Paper No. e6, 2021.

\bibitem[Des23a]{desmet_computations}
Jari Desmet.
\newblock {Computations supporting ``Non-associative Frobenius algebras of type
  $E_7$''}, 11 2023.
\newblock \url{https://github.ugent.be/jardsmet/main-CG-E7}.

\bibitem[Des23b]{desmetE6}
Jari Desmet.
\newblock Non-associative {F}robenius algebras of type {$^1E_6$} with trivial
  {T}its algebras, 2023.

\bibitem[Des23c]{jariprevious}
Jari Desmet.
\newblock Non-associative {F}robenius algebras of type {$G_2$} and {$F_4$}.
\newblock {\em Transformation Groups}, 2023.

\bibitem[DMVC21]{DMVC21}
Tom De~Medts and Michiel Van~Couwenberghe.
\newblock Non-associative frobenius algebras for simply laced {C}hevalley
  groups.
\newblock {\em Trans. Amer. Math. Soc.}, 374:8715--8774, 2021.

\bibitem[DSJ{\etalchar{+}}20]{sagemath}
The~Sage Developers, William Stein, David Joyner, David Kohel, John Cremona,
  and Burçin Eröcal.
\newblock {SageMath}, version 9.0, 2020.

\bibitem[Gar01a]{garibaldie7arbitrary}
R.~Skip Garibaldi.
\newblock Groups of type {$E_7$} over arbitrary fields.
\newblock {\em Comm. Algebra}, 29(6):2689--2710, 2001.

\bibitem[Gar01b]{garibaldie6e7}
R.~Skip Garibaldi.
\newblock Structurable algebras and groups of type {$E_6$} and {$E_7$}.
\newblock {\em J. Algebra}, 236(2):651--691, 2001.

\bibitem[GG15]{GG15}
Skip Garibaldi and Robert~M. Guralnick.
\newblock Simple groups stabilizing polynomials.
\newblock {\em Forum of Mathematics, Pi}, 3:e3, 2015.

\bibitem[Jan03]{jantzenrepalggroups}
Jens~Carsten Jantzen.
\newblock {\em Representations of algebraic groups}, volume 107 of {\em
  Mathematical Surveys and Monographs}.
\newblock American Mathematical Society, Providence, RI, second edition, 2003.

\bibitem[KMRT98]{involutions}
Max-Albert Knus, Alexander Merkurjev, Marksu Rost, and Jean-Pierre Tignol.
\newblock {\em The book of involutions}, volume~44.
\newblock American Mathematical Soc., 1998.

\bibitem[Lü01]{lub01}
Frank Lübeck.
\newblock Small degree representations of finite {C}hevalley groups in defining
  characteristic.
\newblock {\em LMS Journal of Computation and Mathematics}, 4:135–169, 2001.

\bibitem[Mal44]{Malcevconjugacy}
A.~Malcev.
\newblock On semi-simple subgroups of {L}ie groups.
\newblock {\em Bull. Acad. Sci. URSS. S\'{e}r. Math. [Izvestia Akad. Nauk
  SSSR]}, 8:143--174, 1944.

\bibitem[Mil17]{Milne}
J.~S. Milne.
\newblock {\em Algebraic groups}, volume 170 of {\em Cambridge Studies in
  Advanced Mathematics}.
\newblock Cambridge University Press, Cambridge, 2017.
\newblock The theory of group schemes of finite type over a field.

\bibitem[Sch95]{nonassocalg}
Richard~D. Schafer.
\newblock {\em An introduction to nonassociative algebras}.
\newblock Dover Publications, Inc., New York, 1995.
\newblock Corrected reprint of the 1966 original.

\bibitem[SV00]{SpringerVeldkamp}
Tonny~A. Springer and Ferdinand~D. Veldkamp.
\newblock {\em Octonions, {J}ordan algebras and exceptional groups}.
\newblock Springer Monographs in Mathematics. Springer-Verlag, Berlin, 2000.

\bibitem[Tit71]{Tits1971}
J.~Tits.
\newblock Représentations linéaires irréductibles d'un groupe réductif sur
  un corps quelconque.
\newblock {\em Journal für die reine und angewandte Mathematik}, 247:196--220,
  1971.

\end{thebibliography}

\end{document}